\documentclass{amsart}
\usepackage{amsmath, amssymb,epic,graphicx,mathrsfs,enumerate}
\usepackage[all]{xy}

\usepackage{amsthm}
\usepackage{amssymb}
\usepackage{latexsym}
\usepackage{longtable}
\usepackage{epsfig}
\usepackage{amsmath}
\usepackage{hhline}

\DeclareMathOperator{\aut}{Aut}

 \DeclareMathOperator{\frat}{Frat}
 \DeclareMathOperator{\psl}{PSL}

\DeclareMathOperator{\End}{End}

\newtheorem{thm}{Theorem}

 \newtheorem{lemma}[thm]{Lemma}

\numberwithin{equation}{section}

\renewcommand{\footnote}{\endnote}
\newcommand{\ignore}[1]{}\makeglossary

\begin{document}
	\bibliographystyle{amsplain}
	\title[Finite solube groups satisfying the replacement property]{Finite soluble groups satisfying\\ the replacement property}

\author{Andrea Lucchini}
\address{
	Andrea Lucchini\\ Universit\`a degli Studi di Padova\\  Dipartimento di Matematica \lq\lq Tullio Levi-Civita\rq\rq \\email: lucchini@math.unipd.it}

	\begin{abstract}We investigate the finite soluble groups $G$ with the following property (replacement property): for every irredundant  generating set $\{g_1,\dots,g_m\}$ of maximal size and for any $1\neq g\in G$ there
		exists an $i\in \{1,\dots,m\}$ so that $\langle g_1,\dots,g_{i-1},g,g_{i+1},\dots,g_m\rangle=G.$
	\end{abstract}
	\maketitle

\section{Introduction}

Let $G$ be a finite group. We say that a subset
$\Omega= \{g_1,\dots,g_n\}$ of $G$ is irredundant (or independent) if every proper
subset of $\Omega$ generates a proper subgroup of $H.$ Let  $m(G)$ the largest size of an irredundant generating set. An important aspect of vector spaces is the elementary fact that any linearly independent set can replace a segment of a basis. One can try to generalize this notion to arbitrary groups. Instead of looking at bases, one can consider the irredundant generating sets. Also, instead of replacing many elements of the generating set, the focus could  be on replacing a single element. This led D. Collins
and R. K. Dennis to make the following definitions.
A group $G$ satisfies the replacement property
	for the generating sequence $\Omega = \{g_1,\dots,g_n\}$ if, for any $1\neq g,$ there
	exists an $i\in \{1,\dots,n\}$ so that $\Omega^\prime=\{g_1,\dots,g_{i-1},g,g_{i+1},\dots,g_n\}$ generates $G.$
	A group $G$ is said to satisfy the replacement property if it satisfies the
	replacement property for all the irredundant generating sequences of length $m(G).$ 
	
	\
	
	In this short note we investigate the finite soluble groups satisfying the replacement property, giving  some alternative characterizations. Before stating our main result, we need to recall some definitions.
	
	\
	
	A group $G$ is called a $K$-group (a complemented group) if its subgroup lattice
	is a complemented lattice, i.e., for a given $H \leq G$ there exists  $X \leq  G$ such
	that $\langle H, X\rangle=G$ and $H \cap X = 1.$ We refer to \cite[Section 3.1]{sl} for an exposition of the main properties of the $K$-groups.

	\
	
	The  M\"{o}bius function $\mu_G$ on the subgroup lattice of $G$ is the function defined inductively by $\mu_G(G)=1$ and  $\mu_G(H)=-\sum_{K < H}\mu_G(K)$ for any $H < G$. 

\
	
Our main result is the following.
	
\begin{thm}\label{main} Let $G$ be a finite soluble group. The following are equivalent:
\begin{enumerate}
	\item $G$ satisfies the replacement property.
\item $G$ is a $K$-group.
\item $\mu_G(1)\neq 0$. 
\end{enumerate}
\end{thm}

The previous statement does not remain true if one drops the solubility assumption. In \cite{cz} M. Costantini and G. Zacher proved that every finite simple group is a $K$-group. However B. Nachman proved that if $p \not\equiv \pm 1\! \mod 10$ and $p \neq 7$ but $p \equiv 1 \! \mod 8,$ then $\psl(2,p)$ fails
the replacement property (see \cite[Corollary 4.4]{nach}), so not all the finite $K$-groups satisfy the replacement property.
 Whiston and Saxl \cite{ws} noticed that $m(\psl(2,7))=m(\psl(2,11))=4$ and therefore  it follows from \cite[Corollary 4.2]{nach} that $\psl(2,7)$ and $\psl(2,11)$ satisfy the replacement property, however $\mu_{\psl(2,7)}(1)=0$ while  $\mu_{\psl(2,11)}(1)=660.$
 
 \
 
We finally introduce a  stronger property.   A group  $G$ is said to satisfy the strong replacement property if it satisfies the replacement property for all the irredundant sequences (not only for those of maximal size). We prove that only few finite soluble groups satisfies this property.

\begin{thm}\label{strong}A finite soluble group $G$ satisfies the strong replacement property if and only if one of the following occurs.
	\begin{enumerate}
		\item $G$ is an elementary abelian $p$-group.
		\item 	$G=V^t \rtimes  H,$ where $H$ is a cyclic group of prime order and $V$ is a faithful irreducible $H$-module. 
	\end{enumerate} 
\end{thm}

\section{Proofs of Theorem \ref{main} and Theorem \ref{strong}}

We first recall a result from \cite{arch} that will play a relevant role in the proof of Theorem \ref{main}.

\begin{thm}\label{archmin} \cite[Theorem 2]{arch} Let $G$ be a finite soluble group. Then $m(G)$ coincides with the
number of complemented factors in a chief series of $G.$
\end{thm}

\begin{lemma}\label{uno}
Let $G$ be a finite soluble group. If $G$ is a $K$-group, then $G$ satisfies the replacement property.
\end{lemma}

\begin{proof}
We prove the statement by induction on the order. Let $\Omega=\{x_1,\dots,x_m\}$ be an irredundant generating sequence of $G,$ with $m=m(G),$ and let $1\neq g\in G.$ We have to prove that there
exists  $i\in \{1,\dots,m\}$ so that $\langle g_1,\dots,g_{i-1},g,g_{i+1},\dots,g_m\rangle=G.$
The statement is certainly true if $G=1$. So assume $G\neq 1$ and let $N$ be a minimal normal subgroup of $G.$ Since $G$ is a $K$-group, $N$ is complemented in $G$ and it follows from Theorem \ref{archmin} that $m(G)-1=m(G/N).$ Hence $\{x_1N,\dots,x_mN\}$ is a redundant generating sequence of $G/N$, so it is not restrictive to assume, up to reordering the sequence $\{x_1,\dots,x_m\},$ that $G/N=\langle x_1N,\dots,x_{m-1}N\rangle.$ Let $H=\langle x_1,\dots,x_{m-1}\rangle.$ Since $HN=G$ we have two possibilities. 

\noindent a) $H=G.$ In this case $G=\langle x_1,\dots,x_{m-1}\rangle=
\langle x_1,\dots,x_{m-1},g\rangle.$

\noindent b) $H$ is a complement of $N$ in $G.$ Notice that in this case $H$ is a maximal subgroup of $G.$ If $g\neq H,$ then 
$G=\langle H, g\rangle=\langle x_1,\dots,x_{m-1},g\rangle.$ We remain with the case $g\in H.$ By \cite[Lemma 3.1.3]{sl}, $H\cong G/N$ is a $K$-group. By induction, $H$ satisfies the replacement property so, since $m(H)=m(G/N)=m-1,$ we conclude that there exists  $i$ so that $H=\langle x_1,\dots,x_{i-1},g,x_{i+1},\dots,x_{m-1}\rangle$. It follows that $G=\langle H, x_m\rangle=\langle x_1,\dots,x_{i-1},g,x_{i+1},\dots,x_{m-1},x_m\rangle=G.$
\end{proof}

\begin{lemma}\label{due}
	Let $G$ be a finite soluble group. If $G$  satisfes the replacement property, then $G$ is a $K$-group.
\end{lemma}

\begin{proof}
We prove this statement by induction on the order of $G$. Assume that 
$G$ satisfies the replacement properties. The Frattini subgroup $\frat(G)$ of $G$ must be trivial, since a nontrivial element of the $\frat(G)$ cannot replace any element of a generating sequence. This implies that the Fitting subgroup $F=N_1\times \dots \times N_t$ of $G$ is a direct of $t$ minimal normal subgroups of $G$ and has a complement, say $H,$ in $G$ (see for example \cite[5.2.15]{rob}). We claim that $H$ satisfies the replacement property. Let $\mu=m(G/N)$, $\{x_1,\dots,x_\mu\}$ an irredundant generating sequence of  $H$ and $h$ a nontrivial element of $H.$  It can be easily seen that $\Omega:=\{y_1,\dots,y_t,x_1,\dots,x_\mu\}$ is an irredundant generating sequence of $G.$ Moreover, by Theorem \ref{archmin}, $m(G)=\mu+t.$ Since $G$ 
satisfies the replacement properties for $\Omega$, one of elements of $\Omega$ can be replaced by $h$ still obtaining a generating set.
 On the other hand, for every $j\in \{1,\dots,t\},$ the subgroup
$$\langle y_1,\dots,y_{j-1},h,y_{j+1},\dots,y_t,x_1,\dots,x_\mu \rangle =\langle y_1,\dots,y_{j-1},y_{j+1},\dots,y_\mu \rangle H$$ is a complement for $N_i$ in $G.$ So there exists $i\in \{1,\dots,\mu\}$ with 
$$G=\langle y_1,\dots,y_t,x_1,\dots,x_{i-1},h,x_{i+1},\dots,x_{\mu}\rangle.$$ But this implies $\langle x_1,\dots,x_{i-1},h,x_{i+1},\dots,x_{\mu}\rangle=H.$
We have so proved that $H$ satisfies the replacement property, but then by induction $H$ is a $K$-group. By \cite[Theorem 3.1.10]{sl}, $G=FH$ is a $K$-group.
\end{proof}

\begin{proof}[Proof of Theorem \ref{main}]By Lemma \ref{uno} and \ref{due}, (1) and (2) are equivalent. Moreover if follows easily from \cite[Theorem 3.1.12]{sl} that a finite soluble group $G$ is a $K$-group if and only if all the chief factors of $G$ are complemented in $G.$ On the other hand, by \cite[Corollary 3.4]{Hawkes}, if $H_0=1<\dots < H_n=G$ is a chief series
of a finite soluble $G$ and $k_i$ denotes the number of complements
	to $H_i/H_{i-i}$ in $G/H_{i-i},$  then
$\mu_G(1) = (-1)^nk_1k_2\cdots k_n.$ So $\mu_G(1)\neq 0$ if and only if all the chief factors of $G$ are complemented, i.e. if and only if $G$ is a $K$-group.
\end{proof}

\begin{proof}[Proof of Theorem \ref{strong}]
Given a finite group $X$, denote by $d(X)$ be the smallest size of an irredundant generating sequence of $X.$

Assume that $G$ satisfies the strong replacement property. Let $d=d(G)$ and assume $G=\langle g_1,\dots,g_d\rangle.$ Let $1\neq N$ be a normal subgroup of $G$ and let $1\neq n\in N.$ Since $G$ satisfies the replacement property
for the sequence $\{x_1,\dots,x_d\},$
there exists $i$ so that $\{g_1,\dots,g_{i-1},n,g_{i+1},\dots,g_d\}=G.$ This implies $d(G/N)\leq d-1.$ So $G$ has the property that every proper quotient can be generated by $d-1$ elements, but $G$ cannot. The groups with this property have been studied in \cite{dvl}. By \cite[Theorem 1.4 and Theorem 2.7]{dvl} either $G$ is an elementary abelian $p$-group of rank $d$  or there exist a finite vector space $V$ and a nontrivial irreducible soluble subgroup $H$ of $\aut(V)$
such that $d(H)<d$ and $G= (V_1\times \dots \times V_t) \rtimes  H,$
where $t=r(d-2)+1,$  $r$ is the dimension of $V$ over $F=\End_H(V)$ and each $V_i$ is $H$-isomorphic to $V.$ We claim that in the latter case  $r=1$ and $H$ is a cyclic group. If $d=2$, then $H$ is cyclic and $V$ is an absolutely irreducible $FH$-module, and this implies $r=1.$ So assume $d>2$. In this case
$t=r(d-2)+1>r.$ By \cite[Lemma 7.12]{gru} $W=V_{r(d-3)+2}\times \dots \times V_t\cong_H V^r$ is a cyclic $FH$-module. Let $w$ be an 
element of $W$ generating $W$ as an $H$-module, for $1\leq i \leq r(d-3)+1$ let $v_i$ be a nontrivial element of $V_i$ and let $\{h_1,\dots,h_n\}$ be an irredundant generating sequence of $H.$ Clearly $\Omega:=\{v_1,\dots,v_{r(d-3)+1},w,h_1,\dots,h_n\}$ is an irredundant generating sequence of $G.$ Let now $z$ be a nontrivial element of $V_t$. Since $G$ satisfies the replacement property
for $\Omega$, we may replace an element of $\Omega$ with $z$ still obtaining a generating sequence. However if we replace $h_i$ by $w$ we generate the proper subgroup $(V_1\times \dots \times V_t) \rtimes  \langle h_1,\dots,h_{i-1},h_{i+1},\dots,h_t\rangle$ and if we replace $v_j$ by $w$ we generate a complement of $V_j$ in $G$. So it must be
$G=\langle v_1,\dots,v_{r(d-3)+1},z,h_1,\dots,h_n\rangle$, but this implies $W=V_t$ and $r=1.$ In particular $V$ is a faithful absolutely irreducible $FH$-module and consequently $H$ is a subgroup of the multiplicative group of the field $F.$ So $H$ is a cyclic group. Let $h$ by a generator of $H$ and, for each $1\leq i \leq t$, let $v_i$ be a nontrivial element of $V_i.$
The group $G$ satisfies the replacement property
for the irredundant generating sequence $\{v_1,\dots,v_t,h\}$. This implies that if $1\neq h^*\in H,$
then $\langle v_1,\dots,v_t,h^*\rangle=G$ and consequently $\langle h^*\rangle=H.$ But then $H$ must be of prime order. 

Conversely assume that either $G$ is an elementary abelian $p$-group of rank $d$ or	$G=V^t \rtimes  H,$ where $H$ is a cyclic group of prime order and $V$ is a faithful irreducible $H$-module.
We have $d(G)=d$ in the first case, $d(G)=t+1$ in the second one (see for example \cite[Theorem 2.2]{lms}). But then we deduce from Theorem \ref{archmin} that $d(G)=m(G)$, so $G$ satisfies the strong replacement property if and only if it satisfies the replacement property. Moreover all the chief factors of $G$ are complemented, hence, by \cite[Theorem 3.1.12]{sl}, $G$ is a $K$-group. We conclude from Theorem \ref{main} that $G$ satisfies the replacement property.
\end{proof}


\begin{thebibliography}{99}
	\bibitem{cz} M. Costantini and G. Zacher,  The finite simple groups have complemented subgroup lattices, Pacific J. Math. 213 (2004), no. 2, 245-–251.
\bibitem{dvl} F. Dalla Volta and A. Lucchini, Finite groups that need more generators than any proper quotient, {J. Austral. Math. Soc.} Ser. A {64} (1998), no. 1, 82--91.	
\bibitem{Hawkes} T. Hawkes, I. Isaacs and  M. \"{O}zaydin, On the M\"{o}bius function of a finite group, Rocky Mountain Journal of Mathematics \textbf{19} (1989), 1003--1034.	
\bibitem{gru} K. Gruenberg, Relation Modules of Finite Groups, Conference Board of the Mathematical
Sciences Regional Conference Series in Mathematics, No. 25, American Mathematical
Society, Providence, RI, 1976.
\bibitem{arch} A. Lucchini, The largest size of a minimal generating set of a finite group, Arch. Math. (Basel) 101 (2013), no. 1, 1–-8. 
\bibitem{lms}A. Lucchini, M.  Morigi and P. Shumyatsky, Boundedly generated subgroups of finite groups, Forum Math. 24 (2012), no. 4, 875-–887.
\bibitem{sl} R. Schmidt, Subgroup lattices of groups, De Gruyter Expositions in Mathematics, 14. Walter de Gruyter \& Co., Berlin, 1994.
\bibitem{nach}B. Nachman, Generating sequences of $\psl(2,p),$ J. Group Theory 17 (2014), no. 6, 925--945.
\bibitem{rob} D. Robinson,
{A course in the theory of groups,}
Graduate Texts in Mathematics, {80} Springer-Verlag, New York-Berlin, 1982.
\bibitem{ws} J. Whiston and J. Saxl, On the maximal size of independent generating sets of $\psl(2,q),$ J. Algebra 258 (2002), 651--657.
\end{thebibliography}
\end{document}